\documentclass[preliminary]{ems-icm} %% change to `american' if you use American English
\usepackage{mhequ}
\usepackage{microtype}
\usepackage{tikz}
\usepackage{jigsaw}
\usepackage{pgfplots}

\usetikzlibrary{calc}
\theoremstyle{plain}
   \newtheorem{theorem}{Theorem}[section]
   
 \theoremstyle{definition}
   
	\newtheorem{remark}[theorem]{Remark} 

\def\CU{\mathcal{U}}
\def\CV{\mathcal{V}}
\def\CO{\mathcal{O}}
\def\CF{\mathcal{F}}
\def\CB{\mathcal{B}}
\def\C{\mathbf{C}}
\def\R{\mathbf{R}}
\def\N{\mathbf{N}}
\def\Z{\mathbf{Z}}
\def\T{\mathbf{T}}
\def\P{\mathbf{P}}
\let\d\partial
\def\E{\mathbf{E}}
\def\CS{\mathcal{S}}
\def\CC{\mathcal{C}}
\def\bn{\mathbf{n}}
\newcommand{\eqdef}{\stackrel{\mbox{\tiny\textrm{def}}}{=}}
\newcommand{\eqlaw}{\stackrel{\mbox{\tiny\textrm{law}}}{=}}

\DeclareMathOperator{\Var}{Var}
\DeclareMathOperator{\Cov}{Cov}

\def\one{\mathbf{1}}
\def\Wick#1{\mathord{{:}#1{:}}}
\let\eps\varepsilon

\def\even{\mathrm{even}}
\def\odd{\mathrm{odd}}

\def\dash{\leavevmode\unskip\kern0.18em--\penalty\exhyphenpenalty\kern0.18em}
\def\slash{\leavevmode\unskip\kern0.15em/\penalty\exhyphenpenalty\kern0.15em}

\makeatletter

\DeclareRobustCommand{\TitleEquation}[2]{\texorpdfstring{\StrLeft{\f@series}{1}[\@firstchar]$\if%
b\@firstchar\boldsymbol{#1}\else#1\fi$}{#2}}

\makeatother

\let\f\frac

\numberwithin{equation}{section}
\pgfplotsset{compat=1.17}

\begin{document}
\firstpage{1}
\doinumber{210}
\volume{1}
\copyrightyear{2022}

%------
% Insert the title of your paper and (if necessary)
% a short title for the running head.
%------
\title{The work of Hugo Duminil-Copin}
\titlemark{Hugo Duminil-Copin}

%------
% Insert full names of the authors.
% Add further authors as follows:
%  \emsauthor{2}{NAME INCL. FULL FIRST NAME}{NAME WITH FIRST NAME INITIALS}
%  \emsauthor{3}{NAME INCL. FULL FIRST NAME}{NAME WITH FIRST NAME INITIALS}
% etc.
%------
\emsauthor{1}{Martin Hairer}{M.~Hairer}

%------
% Use \authormark if the list of authors is too
% long for the running head: \authormark{A.~Doe et al.}
%------

%------
% Add one \emsaffil and one \email for each author.
%------
\emsaffil{1}{Imperial College London, UK \email{m.hairer@imperial.ac.uk}}

%\dedication{Dedicated to ...}

%------
% Add MSC 2020 codes according to www.ams.org/msc/msc2020.html:
% one primary code (in curly braces) 
% and a list of secondary codes separated by commas (in square brackets)
%------
\classification[82B26, 82B43]{82B20}

%------
% Add a list of keywords.
%------
\keywords{Ising model, Potts model, percolation}

%------
% Insert your abstract.
%------
\begin{abstract}
The past decade has seen tremendous progress in our understanding of 
the behaviour of many probabilistic models at or near their ``critical point''.
On the 5th of July 2022, Hugo Duminil-Copin was awarded the Fields medal for
the crucial role he played in many of these developments. In this short review article,
we will try to put his work into context and present a small selection of his results.
\end{abstract}

\maketitle

%------
% INSERT THE BODY OF THE PAPER HERE (except
% acknowledgments, funding info and bibliography)
%------

%\tableofcontents

%------
% Insert acknowledgments and information
% regarding funding at the end of the last
% section, i.e., right before the bibliography.
%------

\section{Introduction}

Hugo Duminil-Copin was awarded the Fields medal in Helsinki
during the opening ceremony of the 2022 virtual ICM.
In this short note, I will try to put his work into context and to give the reader
a glimpse of why the questions it addresses are not only very interesting from a purely mathematical
perspective, but also contribute to further our understanding of nature at a fundamental level. 
I should start first of all with a disclaimer. Hugo Duminil-Copin is an astounding problem solver
and, while his interest falls squarely into the general area of probability theory and in
particular the type of probabilistic problems that arise when studying microscopic
models for statistical mechanics, I will not be able to do justice to the breadth of his 
contributions. Furthermore, my own area of expertise is somewhat tangential to that of Duminil-Copin,
so this note should be taken as the point of view of an interested outsider. In particular, any
misrepresentations of his results and \slash or techniques will be entirely due to my own ignorance.

In its broadest form, classical statistical mechanics can be thought of as the study of the 
global behaviour of ``large'' systems (of ``size'' $N \gg 1$) that are comprised of many 
identical ``small'' subsystems 
interacting with each other. One typically indexes the subsystems by a discrete
set $\Lambda_N$ with $\lim_{N \to \infty} |\Lambda_N| = \infty$ and one is interested in 
 quantities that are stable as $N \to \infty$. In many cases of interest, one has $\Lambda_N \subset \Lambda$ for $\Lambda$ a discrete subset of Euclidean space (typically a regular lattice) 
 and its elements are
interpreted as a physical location of the corresponding subsystem; the interaction between
subsystems may then depend on their locations. (In most models they actually depend 
only on their relative positions, a notion that generalises very well to locations
taking values in more general symmetric spaces.)

Let us write $S$ for the state space of one single such subsystem, so that the state space for
the full system is $\CS_N \eqdef S^{\Lambda_N}$.
In \textit{equilibrium statistical mechanics}, we furthermore assume that $S$ is equipped 
with a ``reference'' probability measure $\mu$ (think of $\mu$ as being normalised counting measure
if $S$ is a finite set, normalised volume measure if it is a compact manifold, etc) and that our
system is described by an \textit{energy function} $H^{(N)} \colon \CS_N \to \R$, which is typically 
comprised of a contribution for each subsystem, as well as additional interaction terms.
In full generality, one would have something like
\begin{equ}[e:generalH]
H^{(N)}(\sigma) = \sum_{A \subset \CS_N} H_A(\sigma_A)\;,
\end{equ}
where $\sigma_A$ denotes the restriction of $\sigma \in S^{\Lambda_N}$ to $S^{A}$
and the function $H_A$ typically only depends on the ``shape'' of the subset $A$,
so satisfies natural invariance properties under translations and possibly reflections and / or 
discrete rotations. In many classical models, the only non-vanishing terms in \eqref{e:generalH}
are those with $|A| \le 2$.

Given such an energy function, we obtain a probability measure $\mu_{\beta,N}$ on $\CS_N$
by setting
\begin{equ}[e:Gibbs]
\mu_{\beta,N}(d\sigma) = Z_{\beta,N}^{-1} \exp \bigl(- \beta H^{(N)}(\sigma)\bigr) \prod_{u \in \Lambda_N}\mu(d\sigma_u)\;,
\end{equ}
where $Z_{\beta,N}$ is chosen in such a way that $\mu_{\beta,N}(\CS_N) = 1$.
Physically, the parameter $\beta > 0$ appearing in this expression is the inverse of the 
temperature of the system. 
To a large extent, (equilibrium) statistical mechanics is the study of $\mu_{\beta,N}$ as $N \to \infty$
with a particular emphasis on the behaviour under $\mu_{\beta,N}$ of observables that take a
``macroscopic'' (of the order of the size of the domain $\Lambda_N$) or ``mesoscopic'' 
(tending to infinity as $N \to \infty$ but much smaller than $|\Lambda_N|$) number of components of $\sigma$ into account.

\subsection{Bernoulli percolation}

The simplest such example is that of $S = \{-1,1\}$, $H_N = 0$, and $\mu(\{-1\}) = \mu(\{1\}) = \f12$.
Regarding the index set $\Lambda_N$, we consider the case of the even elements of a large box in $\Z^2$,
namely $\Lambda_N = \{u\in \{-N,\ldots,N\}^2\,:\, u_1 + u_2\,\mathrm{even}\}$.
(The reason why we make this strange choice rather than simply taking all elements of $\{-N,\ldots,N\}^2$
will soon become clear.) 

One of the simplest kind of ``global'' observables for this system is given by the following kind of
linear statistics. Given a 
smooth function $\phi \colon [-1,1]^2 \to \R$, we define $I_\phi^N \colon \CS_N \to \R$ by
\begin{equ}[e:localAverages]
I_\phi^N(\sigma) = N^{-\alpha} \sum_{u \in \Lambda_N} \sigma_{u} \phi\big(u/N\big)\;.
\end{equ}
Note that this is exhaustive: for any fixed $N$,
if we know $I_\phi^N(\sigma)$ for every smooth function $\phi$, then
we can in principle recover the argument $\sigma$ itself. A version of the central limit theorem
then immediately yields the following result:

\begin{theorem}\label{theo:CLT}
Setting $\alpha = 1$, the joint distribution of $I_\phi^N(\sigma)$
for any finite collection of test functions $\phi$ as above converges as $N \to \infty$
to the law of a collection of jointly centred Gaussian random variables $I_\phi$ 
such that 
\begin{equ}
\E I_\phi\, I_\psi = \f12\int_{[-1,1]^2} \phi(x)\psi(x)\,dx\;.
\end{equ}
\textup{(The factor $\f12$ appearing here comes from the fact that the local density of $\Lambda_N$ in $\Z^2$ is $\f12$.)}
\end{theorem}

A much more interesting kind of global observables is given by the connectivity properties
of $\sigma$, which were first studied by Broadbent and Hammersley \cite{Percolation}. 
These are however \textit{much} harder to analyse and, even though the model
just described appears at first sight to be somewhat trivial, most of its results already lead us squarely 
into 21st century mathematics. 
In order to describe what we mean by ``connectivity'' in this context, instead of interpreting elements
 $u \in \Lambda_N$ as points in $\Z^2$, we interpret them as nearest-neighbour edges of
a suitable sublattice of $\Z^2$ by associating to $u$ the unique edge 
$e_u$ of $\Z_\even \times \Z_\odd$ with midpoint $u$.
We will also write $e_u^*$ for the edge of $\Z_\odd \times \Z_\even $ with midpoint $u$.
In other words, we set
\begin{equ}
e_u = 
\left\{\begin{array}{cl}
	(u_\downarrow, u_\uparrow) & \text{if $u_1$ is even,} \\
	(u_\leftarrow, u_\to) & \text{if $u_1$ is odd,}
\end{array}\right.\qquad
e_u^* = 
\left\{\begin{array}{cl}
	(u_\leftarrow, u_\to) & \text{if $u_1$ is even,} \\
	(u_\downarrow, u_\uparrow) & \text{if $u_1$ is odd.}
\end{array}\right.
\end{equ}
Here, given $u = (u_1,u_2) \in \Z^2$, we write $u_\leftarrow = (u_1-1, u_2)$, etc.
The endpoints of these edges do belong to the stated sublattices of $\Z^2$ since $u_1 + u_2$ is even, 
so either both $u_1$ and $u_2$ are even or both are odd.

Given a configuration $\sigma \in \CS_N$, we interpret edges $e_u$ with $\sigma_u = -1$ as ``open''
and draw them in black, while the remaining edges are considered ``closed'' and are drawn in light grey.
This yields a picture like shown on the left in Figure~\ref{fig:perco}.
We can then ask for example what is the probability $p_N$ that it is possible to go from the left 
boundary of the light gray graph to the right boundary (the "boundary" here consists of the 
ends of the dangling edges) while only traversing black edges. 
It turns out that this probability does take non-trivial values even for large values for $N$. 
As a matter of fact, it is independent of $N$ as the following classical result
(see for example \cite[Lem.~11.21]{Grimmett}) shows.

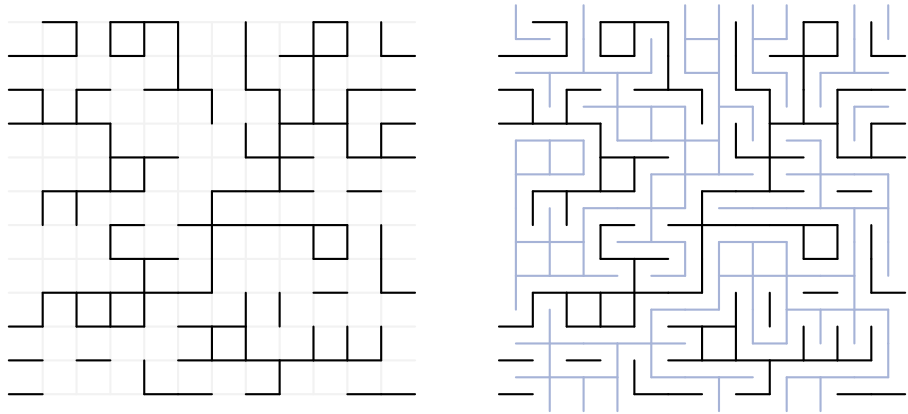
\begin{figure}\begin{center}
\begin{tikzpicture}[every path/.style={thick, black,line cap=round}]
	\pgfmathsetseed{42}%
    \def\h{0.45}
    \draw[white] (0,0.5*\h) rectangle ($(12*\h,12.5*\h)$);
    \foreach \x in {1,2,...,12} {
        \ifnum\x<12 
        	\draw[black!5] ($\h*(\x,1)$) -- ($\h*(\x,12)$);
		\fi
        \draw[black!5] ($\h*(0,\x)$) -- ($\h*(12,\x)$);
    	}

    \foreach \y in {1,2,...,12}{
        \foreach \x in {1,2,...,12}{
            \pgfmathrandominteger{\a}{1}{100}
            \ifnum\a>50 
                \draw ($\h*(\x-1,\y)$) -- ($\h*(\x,\y)$);
            \fi
            \ifnum\y>1
            \ifnum\x<12
                \pgfmathrandominteger{\a}{1}{100}
                \ifnum\a>50 
                    \draw ($\h*(\x,\y-1)$) -- ($\h*(\x,\y)$);
                \fi
            \fi\fi
        }
    }
\end{tikzpicture}%
\hspace{1cm}
\begin{tikzpicture}[every path/.style={thick, black,line cap=round}]
	\pgfmathsetseed{42}%
    \def\h{0.45}
    %
%    \foreach \x in {1,2,...,11} {
%        \draw[black!5] ($\h*(\x,1)$) -- ($\h*(\x,11)$);
%        \draw[black!5] ($\h*(1,\x)$) -- ($\h*(11,\x)$);
%    	}

    \foreach \y in {1,2,...,12}{
        \foreach \x in {1,2,...,12}{
            \pgfmathrandominteger{\a}{1}{100}
            \ifnum\a>50 
                \draw ($\h*(\x-1,\y)$) -- ($\h*(\x,\y)$);
            \else
                \draw[blue!40] ($\h*(\x-0.5,\y+0.5)$) -- ($\h*(\x-0.5,\y-0.5)$);
            \fi
            \ifnum\y>1
            \ifnum\x<12
                \pgfmathrandominteger{\a}{1}{100}
                \ifnum\a>50 
                    \draw ($\h*(\x,\y-1)$) -- ($\h*(\x,\y)$);
                \else
                    \draw[blue!40] ($\h*(\x-0.5,\y-0.5)$) -- ($\h*(\x+0.5,\y-0.5)$);
                \fi
            \fi\fi
        }
    }
\end{tikzpicture}%
\end{center}
\caption{On the left, we draw a typical percolation configuration for $N=11$. On the right,
the same configuration is drawn together with its dual configuration in light blue.}\label{fig:perco}
\end{figure}

\begin{theorem}
One has $p_N = \f12$ for every $N$.
\end{theorem}

\begin{proof}
The trick is to observe that given a configuration $\sigma \in \CS_N$, if we draw the dual configuration
$\sigma^* \in \CS_N$ defined by $\sigma^*_u = -\sigma_u$ by colouring (in blue, say) the edges
$e^*_u$ with $\sigma^*_u = -1$, then we obtain a drawing with the property that blue edges 
never intersect black edges. As a consequence, it is possible to cross the square from left to right by traversing
only black edges if and only if it is \textit{not} possible to cross it from top to bottom by traversing only blue edges.
(See Figure~\ref{fig:perco}.) On the other hand, the law of the collection of blue edges is the same as that 
of the collection of black edges, only rotated by $90^\circ$, so that we must have
$p_N = 1-p_N$ as claimed.
\end{proof}

\begin{remark}\label{rem:critical}
If, instead of choosing edges to be open with probability $\f12$, we choose them to be open 
with some probability $p$, then we have $p_N \to 1$ for $p > \f12$ and $p_N \to 0$
for $p < \f12$. This is an example of \textit{phase transition}: an abrupt change in the
behaviour of some global observables as a parameter of the model is varied continuously.
In this specific example, we were able to determine the \textit{critical value} $p_c = \f12$ 
explicitly by exploiting an exact duality. 
\end{remark}

It is similarly possible to obtain a large collection of interesting global observables by
taking a shape $\CU \subset [-1,1]^2$ diffeomorphic to a square and considering the 
analogous event $A_\CU^{(N)} \subset \CS_N$  
asking whether it is possible to connect the left and right edges of $N \CU$
(without ever leaving $N \CU$) by a path following only open edges of a given configuration
$\sigma \in \CS_N$. Again, the knowledge of these events is an exhaustive statistics for any given
fixed value of $N$. It is furthermore known that for any finite number of such shapes 
$\{\CU_i\}_{i \in I}$ (for $I$ some finite index set) the
random variables $\{[A_{\CU_i}^{(N)}]\}_{i \in I}$ converge in law to a non-degenerate limit 
$\{[A_{\CU_i}]\}_{i \in I}$ as $N \to \infty$ \cite{Stas}. (Here, we write $[A]$ for the indicator function
of an event $A$.)
An amazing fact is that this scaling limit is conformally invariant: if $\phi \colon D \to D'$ is a
conformal map between two smooth simply connected domains $D, D' \subset \C$ such that $[-1,1]^2 \subset D$
and such that $\CV_i \eqdef \phi(\CU_i) \subset [-1,1]^2$, then the joint law of the 
random variables $\{[A_{\CV_i}]\}_{i \in I}$ is \textit{the same} as that of
$\{[A_{\CU_i}]\}_{i \in I}$. 

This conformal invariance turns out to be a crucial feature of 
the scaling limits of many equilibrium statistical mechanics models in two dimensions.
It provides a link to conformal field theory which, at a purely mathematical level, can be thought of
as the study of irreducible representations of the Virasoro algebra.
In particular, it strongly suggests that the possible large-scale behaviours
one can see for two-dimensional equilibrium models come in a one-parameter family of ``universality classes''
parametrised by the central charge of the corresponding conformal field theory.
(In the case of percolation, it turns out that this central charge is given by $c=0$.)

\subsection{The Ising model}

The next-``simplest'' model of statistical mechanics falling into the category of 
equilibrium models described above is the Ising model \cite{Lenz,Ising}. (See also the
review article \cite{HugoICM} in these proceedings which contains a more detailed
account of the various developments spawned by this model.)
In this case, the 
index set is given by $\Lambda_N = \{-N,\ldots,N\}^d$ for some $d \ge 1$, the reference
measure $\mu$ and local state space $S$ are as above, but this time one has
$H_A = 0$ unless $A = \{u,v\}$ with $u,v \in \Z^d$ such that $|u-v| = 1$, in which
case one sets $H_A(\sigma) = -\sigma_u \sigma_v$.
This time, the model has a non-trivial dependence on the parameter $\beta$ appearing 
in \eqref{e:Gibbs}, which plays a role somewhat similar to the parameter $p$
that appeared in Remark~\ref{rem:critical}. 

At a very qualitative level, the situation is somewhat similar to what happened in the case 
for percolation: in every
dimension $d \ge 2$ there exists a critical (dimension-dependent) value $\beta_c$ 
which delineates two different regimes. At ``high temperature'', namely for $\beta < \beta_c$,
the \textit{spontaneous magnetisation}, namely the random quantity 
$N^{-d} \sum_{i\in \Lambda_N} \sigma_i$, converges to $0$ in probability as
$N \to \infty$. For $\beta > \beta_c$ on the other hand, it converges in probability to a
limiting random variable that can take exactly two possible values $\pm h_\beta \neq 0$
with equal probabilities. The actual value of $\beta_c$ is only known in dimension~$2$ 
where it equals $\beta_c = \log \sqrt{1+\sqrt 2}$ \cite{Onsager}. (There is no phase transition
at all in dimension~$1$ and the spontaneous magnetisation always vanishes, so in some 
sense $\beta_c = +\infty$ there.)

It is again possible to ask the same questions as in the case of Bernoulli percolation. 
This time however even the analogue of Theorem~\ref{theo:CLT}, which was an essentially trivial
consequence of the central limit theorem (or at least a version thereof), is already highly
non-trivial. It was shown in a recent series of works \cite{Ising2,Ising3} that if 
one chooses $\beta = \beta_c$ and $\alpha = 15/8$ in the expression \eqref{e:localAverages} in dimension $d=2$,
then it converges in law to non-trivial limiting random variables, jointly for any fixed number of test functions. This time however the limiting distributions are not Gaussian (they actually exhibit 
an even faster decaying tail behaviour).
Note that the exponent $\alpha$ is closely related to the behaviour of $\E_c \sigma_u \sigma_v$ 
(where $\E_c$ denotes the expectation under the Gibbs measure \eqref{e:Gibbs} for the critical value
of the inverse temperature $\beta$) since, assuming that $\E_c \sigma_u \sigma_v \approx |u-v|^{-\delta}$,
one finds that 
\begin{equ}
\E_c \big(I_\phi^N(\sigma)\big)^2
= N^{-2\alpha} \sum_{u,v} \phi(u/N)\phi(v/N) \E_c \sigma_u \sigma_v
\lesssim N^{-2\alpha} \sum_{u,v} |u-v|^{-\delta} \approx N^{2d - (\delta \wedge d) -2\alpha}\;,
\end{equ}
so that one expects the relation $\alpha = d - (\delta \wedge d)/2$, which (correctly)
leads to the prediction $\delta = \f14$. This and a number of other properties of the Ising model at
criticality allow to associate it to the conformal field theory with central charge $c = \f12$.

The picture in higher dimensions is much less clear however. For $d \ge 5$, it was shown 
in \cite{Triv1,Triv2,Triv3} that the correct scaling exponent to use in \eqref{e:localAverages} at 
 $\beta  =\beta_c$ is $\alpha = 1+\f d2$ and that the
limit is a Gaussian Free Field, namely the Gaussian random distribution with correlation function 
given by the Green's function of the 
Laplacian (with Neuman boundary conditions on the square). In dimension $d=3$, virtually nothing
is known rigorously about the critical Ising model, not even the value of its scaling exponents,
although much progress has been made at a non-rigorous level with the development of 
the ``conformal bootstrap'' \cite{Bootstrap1,Bootstrap2}. Regarding the case $d=4$, it was somewhat 
unclear until very recently whether the Ising model at criticality should be ``trivial'' (i.e.\ described
by Gaussian distributions) or not. This was eventually settled by Aizenman and Duminil-Copin in
the work \cite{Phi44} where they show that any subsequential limit for expressions of the form
\eqref{e:localAverages} as $N \to \infty$ (and $\beta \to \beta_c$) must necessarily be Gaussian. 

In fact, some of the results just mentioned are shown for the ``lattice $\Phi^4$ model''
which is the equilibrium model with $S = \R$, as well as
\begin{equ}
H_{\{u\}}(\sigma) = V(\sigma_u) \eqdef  \sigma_u^4 - \alpha\sigma_u^2\;,\qquad
H_{\{u,v\}}(\sigma) = \f12 (\sigma_u - \sigma_v)^2\;,
\end{equ} 
again provided that $u$ and $v$ are nearest-neighbours, and with $c$ an additional parameter.
While this appears to be very different from the Ising model at first sight, we can see that it is
actually a generalisation of it: if the constant $\alpha$ is large, then the potential $V$ has two very deep 
wells with minima located at $\pm \sqrt \alpha$, so its effect is to impose that $\sigma_u \approx \pm \sqrt \alpha$
with high probability. The main contribution then comes from the cross-term of the square in the 
two-body term which is the same as for the Ising model. These kind of considerations lead one to 
expect that, since these models exhibit long-range correlations at the critical temperature
(in the sense that the correlation $\E \sigma_x\sigma_y$ decays slowly in $|x-y|$ as already 
pointed out earlier) which should furthermore lead to some form of self-averaging, the Ising model
and the $\Phi^4$ model exhibit the same behaviour at criticality.

\subsection{A general picture}

The general picture that should by now be emerging from our discussion can be summarised as follows:
\begin{enumerate}
\item Many of the simplest local equilibrium systems do exhibit a phase transition, namely there
exists a critical value $\beta_c$ at which the qualitative large scale behaviour of the system
changes abruptly. In general, a system may depend on additional parameters in which case
one may see a more complicated \textit{phase diagram} with several regions in
parameter space where the global behaviour of the system displays qualitatively different
behaviour. In any case, the ``high temperature / small $\beta$ phase'' is expected to behave
in such a way that what happens in well separated regions of space is very close to independent.

\item In dimension $2$, many of these systems appear to  exhibit a form of conformal invariance 
at criticality, even though no rotation symmetry is built a priori into their description. When this
happens, the link to $2d$ conformal field theories (and the associated probabilistic objects like
SLE \cite{SLE}, QLE \cite{QLE}, etc) provides a hugely powerful machinery to predict \dash and 
in a number of cases also 
rigorously prove \dash their behaviour.

\item The universe of local statistical mechanics models can be subdivided into broad classes of models
that exhibit a shared large-scale behaviour at criticality. These are called ``universality classes''
and, in the $2d$ equilibrium case, they are expected to come in families parametrised by a 
real parameter, the central charge. (For certain values of the central charge, one expects to have 
several ``subclasses'', but we will not discuss this kind of subtlety here.)

\item Although one still expects conformal invariance at criticality in higher dimensions, this is a much smaller
symmetry there and therefore appears to provide somewhat less insight\footnote{See however the recent 
breakthrough made in the approximation of the critical exponents of the $3d$ Ising model using
the ``conformal bootstrap'' \cite{Bootstrap1,Bootstrap2} already mentioned above.}. One also expects the situation there to
be more rigid than in two dimensions, with fewer universality classes. (Possibly only a discrete family.)

\item Models that have ``obvious'' variants in every dimension typically have a critical dimension
above which their behaviour at criticality is ``trivial'' in the sense that it exhibits Gaussian
behaviour. (Typically with correlation function given by the Green's function of the Laplacian.)
In the case of the Ising universality class, this critical dimension is $4$, while in the case of
Bernoulli percolation it is $6$. 
\end{enumerate}

One important branch of modern probability theory aims to put this general picture onto rigorous
mathematical footing. The remainder of this article is devoted to a short overview of some of 
Hugo Duminil-Copin's many contributions to this vast programme.
This represents of course a mere sliver of his work and completely ignores very substantial chunks 
of it. By presenting not just a long laundry list of results that he proved and conjectures that he
settled but instead an overview of the strategy of proof for a few select results, I hope to be able
to convey one of the features of Duminil-Copin's body of work, namely that he has a knack for
finding just the right way of looking at a problem that had hitherto been overlooked. In many
cases, this only provides small cracks in the problem's armour that still require tremendous technical 
skill to be wedged open, but in some cases it results in surprisingly simple but ingenious proofs.
Either way, I am very much looking forward to learning more from Duminil-Copin's insights for many
years to come.

\section{(Dis)continuity of phase transitions}
\label{sec:Disc}

One very natural question in this area is whether one can take the limit $N \to \infty$ in 
\eqref{e:Gibbs}. At this stage, we note that the definition of $H^{(N)}$ given in
\eqref{e:generalH} is not necessarily the most natural one since it restricts the sum over those clusters
$A$ that are constrained to \textit{entirely} lie in $S_N$. Another possibility 
that appears just as natural would be to restrict the sum over clusters that merely 
intersect $S_N$, but to specify some fixed ``boundary condition'' $\bar \sigma \in S^\Lambda$ that 
is used to compute the values of the $H_A$ with $A$ intersecting both 
$\Lambda_N$ and $\Lambda \setminus \Lambda_N$ in the sense that we interpret
$\sigma_A$ in \eqref{e:generalH} as $\sigma_{A,x} = \sigma_x$ for $x \in A \cap \Lambda_N$ and
$\sigma_{A,x} = \bar \sigma_x$ otherwise.
\begin{figure}
\begin{center}
\includegraphics[width=5cm]{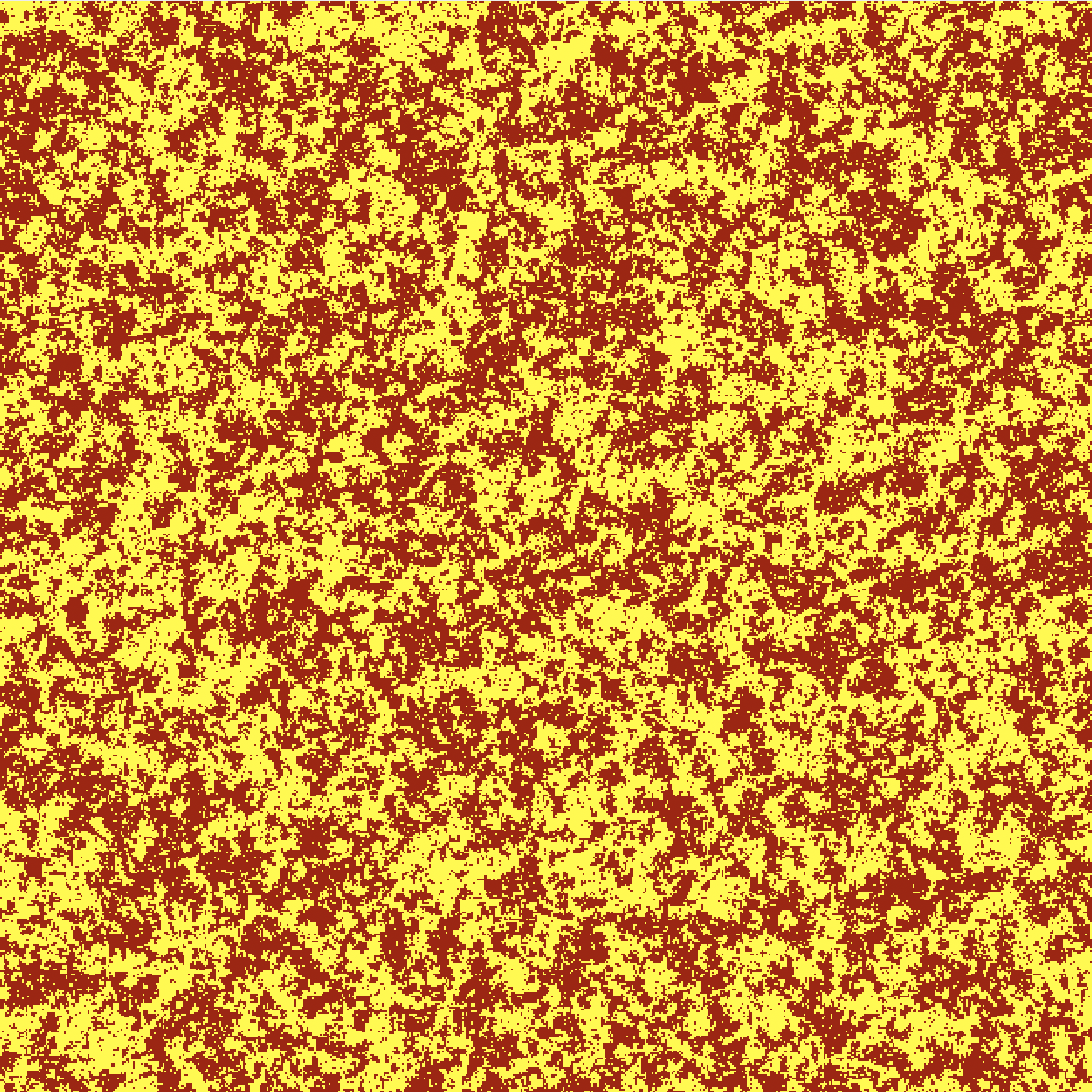}
\includegraphics[width=5cm]{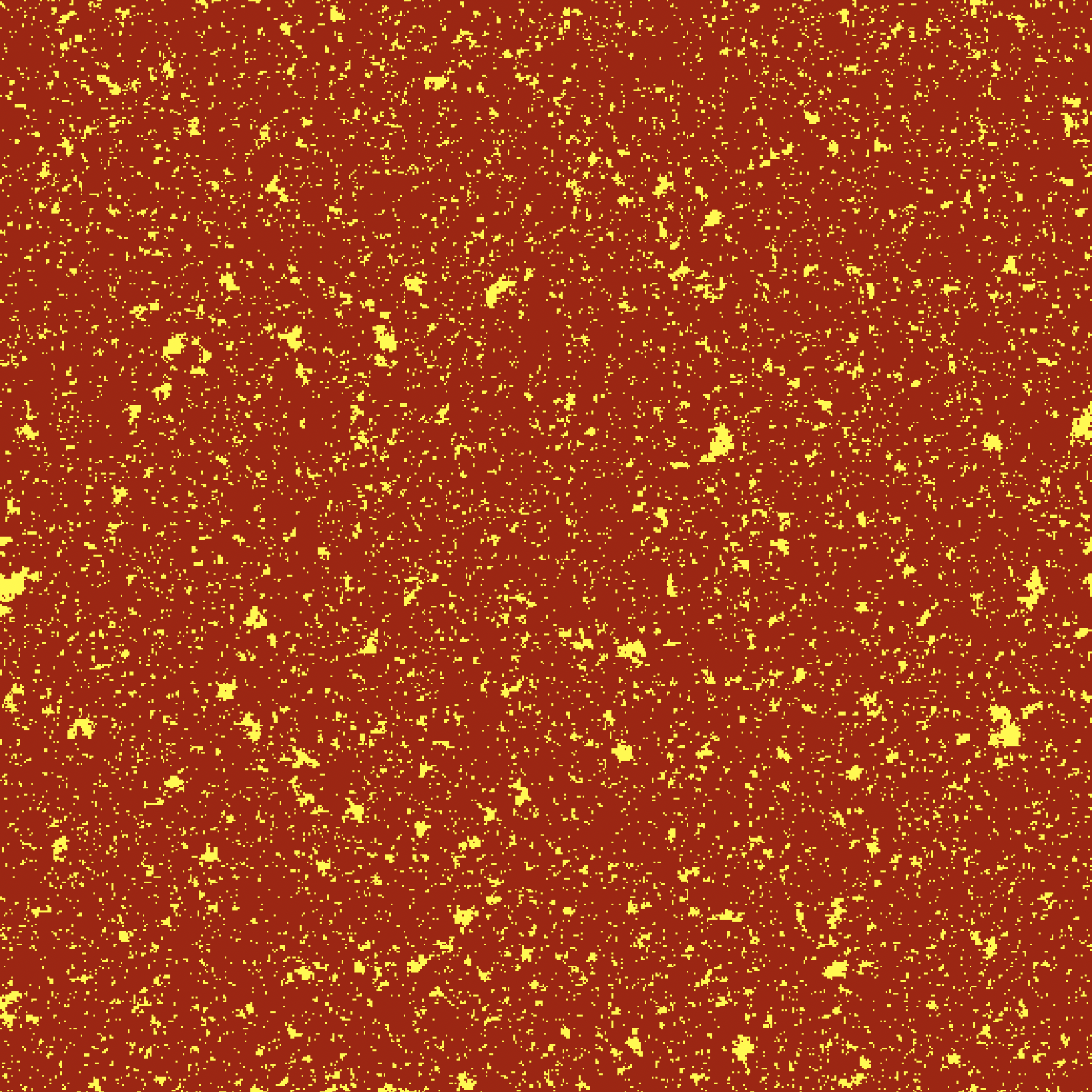}
\end{center}
\caption{Typical Ising configurations for $\beta < \beta_c$ (left) and $\beta > \beta_c$ 
(right).}\label{fig:Ising}
\end{figure}

In many examples of interest (including the case of the Ising model, but \textit{not} the case
of percolation), the measure $\mu_\beta = \lim_{N \to \infty} \mu_{\beta,N}$ is well-defined
(i.e.\ independent of the choice of boundary condition) for $\beta < \beta_c$ while one can obtain
several distinct limits in the case $\beta > \beta_c$. Figure~\ref{fig:Ising} shows typical samples
drawn from $\mu_\beta$ for the Ising model with $\bar \sigma \equiv 1$. In the case $\beta > \beta_c$, 
the resulting sample clearly ``remembers''
 the bias introduced by $\bar \sigma$ in the sense that a typical configuration consists of a ``sea'' 
 of spins taking the dominant value $+1$ (brown) with small ``islands'' of spins taking the value $-1$ (yellow). Had we set $\bar \sigma \equiv -1$, we would have obtained a sample with
 the opposite behaviour, which illustrates the non-uniqueness of the infinite-volume measure 
 $\mu_\beta$ in this case. 
In the case  $\beta < \beta_c$ on the other hand, each one of the two possible spin 
 values is about equally 
represented and the measure is symmetric under the substitution $1 \leftrightarrow -1$, which 
illustrates the uniqueness of $\mu_\beta$.
It is in fact a theorem in the case of the Ising model that for $\beta > \beta_c$ 
there exist exactly two translation invariant infinite volume
measures $\mu_\beta^\pm$ corresponding to boundary conditions $\bar \sigma \equiv \pm 1$ and
that every accumulation point of $\mu_{\beta,N}$ for any sufficiently homogeneous 
boundary condition as $N \to \infty$ is 
a convex combination of $\mu_\beta^+$ and $\mu_\beta^-$.

This raises the question of the uniqueness of $\mu_\beta$ at $\beta = \beta_c$. If it is, then we say
that the phase transition is \textit{continuous}, otherwise it is said to be \textit{discontinuous}.
The reason for this terminology is that continuity in this sense turns out to be equivalent to the 
continuity of the maps $\beta \mapsto \mu_\beta^\pm$ at $\beta = \beta_c$.
It has been known for quite some time \cite{ContTwo,ContHigh} that the phase transition 
for the Ising model is 
continuous in dimensions $d = 1,2$ as well as $d \ge 4$. The reason why dimensions $1$ and $2$ are 
typically much better understood is that the Ising model is ``solvable'' in these dimensions in the
sense that explicit expressions can be derived for the expectation of a large number of observables
under $\mu_{\beta,N}$ (this solution is straightforward in $d=1$ \cite{Ising} where no phase
transition is present, but it was a major 
breakthrough when Onsager obtained his exact solution for $d=2$ \cite{Onsager}). 
Dimension $d = 4$ on the other hand is the ``upper critical dimension'' beyond which
 the model is expected to be 
``trivial'' (i.e.\ described by Gaussian random variables in the scaling limit) 
which allows to use a number of powerful techniques, including
for example the \textit{lace expansion} \cite{Lace,LaceIsing}.

This leaves the case $d=3$ which is of course the physically most interesting one 
since the Ising model is a toy model of ferromagnetism and its dimensions represent the usual
spatial dimensions. Heuristic considerations suggest that the phase transition is also continuous
there, and this is consistent with physical experiments, assuming that the Ising model 
belongs to the same universality class as that of a genuine physical magnet. 
In the article \cite{ContIsing}, Duminil-Copin et al.\ gave the first rigorous proof that 
this is indeed the case. The proof relies on the introduction of the quantity
\begin{equ}
M(\beta) = \inf_{B \subset \Z^3} {1\over |B|^2} \sum_{x,y \in B} \int \sigma_x\sigma_y\,\mu^0_\beta(d\sigma)\;,
\end{equ}
where $\mu_\beta^0$ denotes the infinite volume limit obtained from
using ``free'' conditions, as well as three main steps. First, they rely on results of \cite{Infrared,Reflect} 
to argue that 
the Fourier transform of $x \mapsto \int \sigma_0\sigma_x\,\mu_\beta^0(d\sigma)$ belongs to
$L^1$ at $\beta = \beta_c$, which implies that $M(\beta_c) = 0$. Then, and this is the main
step, they show that 
having $M(\beta) = 0$ implies that a certain percolation model with long-range correlations
constructed from the Ising model admits no infinite clusters. Finally, they use a variant
of the ``switching lemma'' \cite{Switching} to show that the quantity 
$\int \sigma_0\sigma_x\,\mu_\beta^+(d\sigma)-\int \sigma_0\sigma_x\,\mu_\beta^0(d\sigma)$ is 
dominated by an explicit function times the probability 
of the origin belonging to an infinite cluster in the above mentioned model and therefore
has to vanish at $\beta = \beta_c$. Once this is known, it is not too difficult to show that 
the spontaneous magnetisation of the Ising model at criticality  must vanish
(namely one has $\int \sigma_0\,\mu_{\beta_c}^+(d\sigma)=0$), which in turn 
yields the desired uniqueness statement.

To illustrate the fact that continuity of the phase transition, whatever the dimension, is
a rather non-trivial property that isn't necessarily expected in general, a good example
is that of the Potts model \cite{Potts}. This is defined similarly to the Ising model,
but this time the local state space $S$ is given by $S= \{1,\ldots,q\}$ for some $q \ge 2$
endowed again with the normalised counting measure as its reference measure. As in the
Ising model, one sets  $H_A = 0$ unless $A = \{u,v\}$ with $u,v \in \Z^d$ such that 
$|u-v| = 1$, in which case one sets $H_A(\sigma) = \one_{\sigma_u = \sigma_v}$. For $q=2$
this is equivalent to the Ising model since their energy functionals only differ by a constant.
Let us also remark that there is an essentially equivalent model called the random cluster 
model (or sometimes the FK model after Fortuin and Kasteleyn who introduced it in \cite{FK}) in which one
directly considers partitions of $\Z^d$ into connected ``clusters'' 
(which one should think of as the edge-connected components of the 
sets $\{u\,:\, \sigma_u = i\}$ for $i \in S$ and a given configuration $\sigma$ of the
Potts model) and which makes sense also for non-integer values of $q \ge 1$.
(In the case $q=1$ the FK model actually reduces to regular Bernoulli percolation.)
See \eqref{e:defFK} below for a more precise definition of this model.

It was conjectured by Baxter in the 70's \cite{Baxter1,Baxter2} that the Potts model on $\Z^2$ exhibits 
a continuous phase transition if and only if $q \le 4$. The pair of
articles \cite{PottsDisc,PottsCont} by Duminil-Copin et al.\ provides proofs of both directions 
of this conjecture. For the sake of brevity we will not comment on the proofs in any detail, but
we note that the proof of continuity of the phase transition for $q \le 4$ is almost completely disjoint
from that in the case of the $3d$ Ising model. A milestone is again to show that the
model at criticality with boundary condition set to one fixed element of $S$ 
admits no infinite cluster. However both the
proof of this fact (exploiting a form of discrete holomorphicity of certain cleverly chosen
observables) and the proof of its equivalence with the uniqueness of the infinite-volume measure 
at criticality (actually they show equivalence of a list of $5$ quite distinct properties which are
of independent interest for the study of the critical Potts model) are completely different.

Regarding the proof of \textit{dis}continuity when $q > 4$, the main tool is a close relation,
first discovered by Temperley--Lieb \cite{TL} in a restricted context and 
then by Baxter et al.\ \cite{Baxter_1976}
in more generality, between the FK model on $\Z^2$ and the so-called six-vertex model.
Configurations of the latter can be visualised as jigsaws where one assigns to each vertex of $\Z^2$ 
(or a subset thereof) one of the six (oriented) tiles
\begin{center}
	\tile[lightgray]{1}{-1}{-1}{1} \qquad
	\tile[lightgray]{-1}{1}{1}{-1} \qquad
	\tile[lightgray]{1}{1}{-1}{-1} \qquad
	\tile[lightgray]{-1}{-1}{1}{1} \qquad 
	\tile[violet]{1}{-1}{1}{-1} \qquad
	\tile[violet]{-1}{1}{-1}{1}
\end{center}
and one enforces the admissibility constraint that the tiles fit together seamlessly. One further postulates
that the probability of seeing a given admissible configuration is proportional to $c^{\# p}$,
where $\# p$ denotes the number of purple tiles in the configuration and $c$ is some fixed constant.
The relation between the six-vertex model and the critical FK model holds for the
specific choice $c = \sqrt{2+\sqrt q}$. The advantage gained from this relation is that the 
six-vertex model is ``solvable'' in a certain sense using the transfer matrix formalism.
This doesn't get one out of the woods since the transfer matrices $V_N$ involved are very large: they
act on a vector space of dimension $2^N$, but are block diagonal with each block
$V_N^{[n]}$ acting on a subspace of dimension $\binom{n}{N}$. Each of these blocks 
is irreducible with positive entries and therefore admits a Perron--Frobenius vector. 
The main technical result of \cite{PottsDisc} is a very sharp asymptotic for the 
Perron--Frobenius eigenvalues of $V_N^{[N/2-r]}$ for fixed $r$ as $N \to \infty$.
Interestingly, the authors are able to prove that the ratios between 
these values converge to finite (and explicit, at least as explicit convergent series) 
limits as $N \to \infty$ and that the values themselves diverge exponentially in $N$
with known exponent, but the common lower-order behaviour of that divergence is not known.
This asymptotic is however sufficient to obtain good control over the partition function of the 
six vertex model and to exploit it to compute an explicit expression for the inverse 
correlation length of the critical Potts model with free boundary conditions when $q > 4$. 
The finiteness of that expression finally allows to deduce the discontinuity of the phase
transition.

To conclude this section, I would like to mention the beautiful article \cite{Sharp} which, 
although not quite
dealing with the question of continuity of the phase transition, does have a related flavour.
The question there is that of the ``sharpness'' of the phase transition which in this particular case
is couched as the question whether it is really true that the measure $\mu_\beta$ has exponentially
decaying correlations (in the sense that the covariance between $f(\sigma_0)$ and $f(\sigma_x)$
decays exponentially fast as $|x| \to \infty$ for any ``nice enough'' function $f \colon S \to \R$)
for \textit{every} $\beta < \beta_c$ and not just for small enough values where a perturbation
argument around $\beta = 0$ (where $f(\sigma_0)$ and $f(\sigma_x)$ are independent under $\mu_0$ as soon as
$x \neq 0$) may apply. One difficulty with this type of statements is that one will
in general not know any closed-form expression for $\beta_c$: in the case of the FK model on the 
square lattice such an expression can be derived by a duality argument \cite{MR2948685}, 
but it is not known for more general situations. The main result of 
\cite{Sharp} is that the phase transition of the FK model on \textit{any} vertex-transitive infinite graph
is sharp.

The main tool in their proof is a novel and far-reaching generalisation of the OSSS inequality 
\cite{OSSS}. The context here is that of increasing random variables $f \colon \{0,1\}^E \to [0,1]$ 
(for a finite set $E$ and for the natural coordinate-wise partial order on $\{0,1\}^E$) 
where $\{0,1\}^E$ is furthermore equipped with a probability measure $\P$ that is itself
\textit{monotonic} in the sense that for every $F \subset E$ and every $e \in E \setminus F$, 
the conditional probabilities $\P(w_e = 1\,|\, \CF_F)$
are increasing functions. (Here $\CF_F$ denotes the $\sigma$-algebra generated by the evaluations
$w \mapsto w_e$ for $e \in F$.) One then considers \textit{any} algorithm that reveals one by 
one the values of an input 
$w \in \{0,1\}^E$ in such a way that the coordinate to be revealed next depends in a deterministic
way on the information gleaned from the revealement up to that point. (In particular, the first 
coordinate to be revealed is always the same since no information has been obtained yet at that point.) 
The algorithm stops once
the revealed values are sufficient to determine the value of $f(w)$, thus yielding a random
set $\hat E \subset E$ of revealed values. The result of \cite{Sharp} is then that one has the inequality
\begin{equ}[e:CorInequ]
\Var(f) \le \sum_{e \in E} \P(e \in \hat E)\, \Cov(f,w_e)\;,
\end{equ}
which looks formally the same as the result of \cite{OSSS}, but the assumption there was that
the measure $\P$ is simply the uniform measure. Since the latter is clearly monotonic (it is such 
that $\P(w_e = 1\,|\, \CF_F)$ is constant), the results of \cite{OSSS} follow as a special case.

Using this result, \cite{Sharp} then obtain the following dichotomy which yields the desired sharpness 
statement.

\begin{theorem}\label{theo:sharp}
Let $G$ be any transitive graph and let $\P_{\beta,n}$ be the FK measure on the ball $\Lambda_n$ of radius
$n$ in $G$. Then, there exists $\beta_c \in \R$ such that, for every $\beta < \beta_c$ 
there exists $c_\beta > 0$ such that $\P_{\beta,n}(0 \leftrightarrow \d \Lambda_n) \lesssim e^{-c_\beta n}$, uniformly in $n$. For $\beta > \beta_c$ on the other hand, there exists $c > 0$ such that 
$\P_{\beta,n}(0 \leftrightarrow \d \Lambda_n) \ge c \min\{1,\beta - \beta_c\}$.
\end{theorem}

Once \eqref{e:CorInequ} is known, the proof is surprisingly simple and relies on two ingredients. First, one can show that the measures $\P_{\beta,n}$
and the function $\one_{0 \leftrightarrow \d\Lambda_n}$ satisfy the assumptions of \eqref{e:CorInequ}. 
Setting $\theta_n(\beta) = \P_{\beta,n}(0 \leftrightarrow \d \Lambda_n)$, a clever choice of search 
algorithm for the (potential) cluster connecting the origin $0$ to $\d \Lambda_n$
then allows to show that one has the bound
\begin{equ}[e:ineq]
\theta_n'(\beta) \gtrsim \sum_{e \in E} \Cov_\beta(\one_{0 \leftrightarrow \d\Lambda_n} ,w_e) \ge {n \over 8\Sigma_n(\beta)}\theta_n(\beta) (1-\theta_n(\beta))\;. 
\end{equ}
where $\Sigma_n = \sum_{k=0}^{n-1} \theta_n$. The fact that the first inequality holds is known and 
can be checked in an elementary way. The second fact is that \textit{any} sequence of functions $\beta \mapsto \theta_n(\beta)$
satisfying a differential inequality of the form \eqref{e:ineq} necessarily satisfies a dichotomy of the type
appearing in the statement of Theorem~\ref{theo:sharp}. Since we are not interested in the regime where $\theta_n$ 
is large, we can rewrite \eqref{e:ineq} as $\theta_n' \ge {cn \over\Sigma_n} \theta_n$.
The fact that the $\theta_n$ then should satisfy such a dichotomy is quite clear: if $\beta$ is such that 
they converge to a non-vanishing limit $\theta$, then $\Sigma_n /n \sim \theta$ and one must have $\theta' \ge c$.
If on the other hand they converge to $0$ on a whole interval $[a,b]$, then that convergence must take place sufficiently 
fast so that $\Sigma_n /n \gg \theta_n$ (since otherwise the previous argument applies).
Since $\Sigma_n / n \sim \theta_n$ for $\theta_n \sim n^{-\alpha}$ as soon as $\alpha < 1$, it 
is then plausible that for any $c < b$ one has 
$\theta_n \ll n^{-1/2}$ (say), implying $\theta_n' \gtrsim \sqrt n \theta_n$ and therefore $\theta_n 
\lesssim e^{-\sqrt n (c-\beta)}$ for $\beta < c$. This shows that $\Sigma_n$ is bounded for $\beta < c$, leading to $\theta_n' \gtrsim  n \theta_n$ and therefore an exponentially (in $n$) small bound as claimed.

\section{Triviality of \TitleEquation{\Phi^4_4}{Phi44}}

It has been known since the groundbreaking work of Osterwalder and Schrader \cite{OS, OS2} that,
at least in some cases, the construction of a (bosonic) quantum field theory satisfying the 
Wightman axioms is equivalent to the construction of a probability measure on the 
space of distributions 
satisfying a number of natural properties. One of the pinnacles of that line of enquiry
was the construction in the seventies of the $\Phi_2^4$ and $\Phi_3^4$ measures
\cite{Nelson1,Nelson2,Simon,MR0231601,MR0359612,MR0408581,MR0384003,MR0416337},
which corresponds to the simplest case of an interacting theory in two or three space-time dimensions
with one type of boson. 

At a heuristic level, the $\Phi_d^4$ measure is the measure $\mu^{(d)}$ on the space of 
Schwartz distributions $\CS'(\R^d)$ (or on the $d$-dimensional torus) given by 
\begin{equ}
\mu^{(d)}(d\Phi) = Z^{-1} \exp \Bigl(- \f12 \int \bigl(|\nabla \Phi(x)|^2 - C \Phi^2(x) + \Phi^4(x)\bigr)\,dx\Bigr)\, d\Phi\;,
\end{equ}
where ``$d\Phi$'' denotes the infinite-dimensional Lebesgue measure on
$\CS'(\R^d)$. This expression is of course problematic at many levels: 
infinite-dimensional Lebesgue measure does not exist, distributions cannot be squared, etc.
If it were only for the term $|\nabla\Phi|^2$, one 
could reasonably interpret this expression as the Gaussian measure $\mu_0$ with covariance operator 
given by the Green's function of the Laplacian, which is a well-defined probability measure
(modulo technicalities arising from the constant mode which can easily be fixed). The measure $\mu_0$
is called the Gaussian Free Field (GFF) since it corresponds to a quantum field theory 
in which particles are free, i.e.\ do not interact with each other at all.

This suggests that a more refined interpretation of the $\Phi^4_d$ measure could be given by 
\begin{equ}[e:defNaive]
\mu^{(d)}(d\Phi) = Z^{-1} \exp \Bigl(- \f12 \int \Phi^4(x)\,dx \Bigr)\,\mu_0(d\Phi)\;.
\end{equ}
This is still ill-defined since the GFF is supported on distributions rather than
functions for any dimension $d \ge 2$. However, setting $\Phi_\eps = \rho_\eps \star \Phi$, the \textit{Wick power}
\begin{equ}[e:Wick]
\Wick{\Phi^4} = \lim_{\eps \to 0} \big(\Phi_\eps^4 - 3 \Phi_\eps^2 \E \Phi_\eps^2\big)\;,
\end{equ}
turns out to be a well-defined random Schwartz distribution  
(i.e.\ the limit exists and is independent of the choice of $\rho_\eps$) 
in dimensions $d < 4$. In dimension $2$, Nelson showed in \cite{Nelson1} that the Radon--Nikodym factor
appearing in \eqref{e:defNaive} with $\Phi^4$ replaced by $\Wick{\Phi^4}$ yields an integrable
random variable, thus leading to a definition of $\mu^{(2)}$. In particular, the $\Phi^4_2$ measure
is equivalent to the GFF.
In dimension $3$, this turns out not to be the case, but it is still possible to show that the measure
\begin{equ}[e:defCorrect]
\mu^{(3)}(d\Phi) = \lim_{\eps \to 0} Z_\eps^{-1} \exp \Bigl(- \f12 \int \Phi_\eps^4(x) - C_\eps \Phi_\eps^2(x)\,dx \Bigr)\,\mu_0(d\Phi)\;,
\end{equ}
is well-defined for a suitable choice of the constant $C_\eps$ which differs from the choice
$3\E \Phi_\eps^2 \sim \eps^{-1}$ suggested by \eqref{e:Wick} by a logarithmically divergent term. (An alternative construction 
of this measure was recently obtained by completely different techniques in 
\cite{Hai14,HM18,MW18a}.)

This discussion begs the question of what happens for $d \ge 4$ and especially when $d=4$ which
is the physically most interesting case from the QFT perspective (remember that dimension here
corresponds to space-time). Regarding the case $d > 4$, it was already shown in the eighties
by Aizenman and Fröhlich \cite{Triv1,Triv2,Triv3} that pretty much any ``reasonable'' 
definition of the $\Phi^4_d$ measure actually coincides with the GFF. This still left the case
$d=4$ which has always been expected to be the hard case since it is ``critical'' in the
sense that, at least at a formal level, the terms $\Phi^4$ and $|\nabla \Phi|^2$ scale in the same
way in the following sense. Writing $\CS_\lambda$ for the transformation
$
(\CS_\lambda F)(x) = F(\lambda x)
$,
the GFF has the self-similarity property $\CS_\lambda \Psi \eqlaw \lambda^{2-d\over 2} \Psi$
for $\Psi$ drawn from $\mu_0$. Pretending that $\Psi$ behaves like a function (even though it
really is a random distribution), we deduce that
\begin{equ}
\CS_\lambda |\nabla \Psi|^2 =  \lambda^{-2} |\nabla \CS_\lambda \Psi|^2
\eqlaw \lambda^{-d} |\nabla \Psi|^2\;,\qquad
\CS_\lambda (\Psi^4) = (\CS_\lambda \Psi)^4 \eqlaw \lambda^{4-2d} \Psi^4\;.
\end{equ}
These exponents are indeed equal if and only if $d=4$. A heuristic calculation actually
suggests that, at higher order, the term $|\nabla \Psi|^2$ dominates the term $\Psi^4$ at
large scales. Variants of this observation have been made rigorous in a number of works
\cite{MR771258,MR892925,MR882810}, including most recently in an impressive series of works by 
Bauerschmidt--Brydges--Slade (see \cite{BBS2,BBS1} and the references therein). 

One way of formulating one of their main results is the framework given in our introduction with 
$S = \R$, $\mu$ being Lebesgue measure, $H_{\{x\}}(\phi) = \f g4 \phi_x^4 + \f\nu2 \phi_x^2$,
$H_{\{x,y\}}(\phi) = |\phi_x - \phi_y|^2$ when $x$ and $y$ are neighbouring lattice sites
in $\Z^4$, and $H_A = 0$ otherwise. This model behaves in a way that is very similar to the Ising
model, to which it degenerates in the regime $g \to \infty$ and $\nu = -g$. Traditionally,
one considers the $\Phi^4$ model with $\beta = 1$, since one can always reduce oneself to this case
by adjusting $g$ and $\nu$, and possibly rescaling the $\phi_x$'s by a factor. 
One typically also considers $g$ fixed,
it is therefore the parameter $\nu$ that is tuneable and plays the role of a ``temperature'' in 
this model. Just like the Ising model, it exhibits a phase transition at some value $\nu_c \in \R$: 
for $\nu > \nu_c$, there 
exists a unique infinite volume measure which is symmetric under $\phi \mapsto -\phi$. For
$\nu < \nu_c$ on the other hand, one finds two distinct infinite-volume measures (as well as their
convex combinations) depending on the boundary conditions one chooses.

A state $\phi \in S^{\Lambda_N}$ with $\Lambda_N = \{-N,
\ldots,N\}^4$ is viewed as a distribution $\iota \phi$ on the torus (of size $2$) by setting, for
every smooth test function $f \colon \T^4 \to \R$,
\begin{equ}
\bigl(\iota \phi\bigr)(f) = \sum_{x\in \Lambda_N} \sigma_N \phi_x f(x/N)\;,
\end{equ}
for a sequence of values $\sigma_N$ chosen in such a way that 
$\E \bigl((\iota \phi)(1)^2\bigr) = 1$. It is then shown in \cite{BBS2} that if $g$ is sufficiently small
and $\nu$ is chosen in a suitable way (close but not quite equal to the critical value $\nu_c$),
then $\iota \phi$ converges to a massive GFF, namely the Gaussian field with
covariance given by $(m^2 - \Delta)^{-1}$ for some $m \in \R$ (which depends on the 
specific way in which $\nu$ is being tuned to approach $\nu_c$ as $N \to \infty$).

While this result strongly suggests that there exists no non-trivial $\Phi^4_4$ measure, it doesn't
rule out the possibility of having a non-trivial scaling limit for the discrete field we
just described at (or near) criticality when the constant $g$ is sufficiently large
(in other words ``at strong coupling''). The technique of proof of \cite{BBS2}
was to implement a rigorous version of the ``renormalisation group technique'', which relies on
a subtle analysis of the behaviour of the renormalisation map near the fixed point given by
the GFF. This is unfortunately perturbative in nature and so has little hope of being
able to deal with arbitrary $g$.
In the recent work \cite{Phi44} however, Aizenman and Duminil-Copin finally
succeeded in showing the following result.

\begin{theorem}\label{theo:Gauss}
For \textit{every} way of adjusting $g = g_N$ and $\nu = \nu_N$ as $N \to \infty$ such that 
$\nu_N \ge \nu_{c,N}$, 
every $M_N \to \infty$ with
$1 \ll M_N \ll N$ and every smooth compactly supported test function $f$, the law of 
$\xi_N^f = \sum_{x\in \Lambda_N} \phi_x f(x/M_N)$, normalised so that its variance is one, 
converges to a normal distribution. 
\end{theorem}

\begin{remark}
The condition $\nu_N \ge \nu_{c,N}$ can actually be slightly relaxed but not too much. This is because,
in the ``low temperature'' regime and with free (or periodic) boundary conditions, one would expect 
the law of $\xi_N^f$ to converge to a Bernoulli random variable rather than a Gaussian.
\end{remark}

At a high level, the main reason why \cite{Phi44} can deal with arbitrary couplings is that 
one can think of their setting as being more akin to ``perturbing around $g=\infty$'' rather than around
$g=0$. In the setting of the introduction, they start by considering the Ising model as
described there (i.e.\ with $\mu = \f12 (\delta_1 + \delta_{-1})$, but then expand their class of models to 
allow for each site to represent a collection of spins with arbitrary ferromagnetic interactions within
a site, instead of a single spin. This has the effect of replacing $\mu$ by any measure that 
can be obtained as the law of $\delta \sum_{i=1}^K s_i$ for some $\delta > 0$ and $K \in \N$, 
and where the $s_i \in \{-1,1\}$ are random variables with a joint distribution proportional to
$\exp(-\sum_{ij} a_{ij} s_is_j)$ for some arbitrary but \textit{positive} coefficients $a_{ij}$.
As was shown
already in the 70's \cite[Thm~1]{SG}, all probability measures on $\R$ of the type $Z^{-1}\exp(c x^2 - gx^4)\,dx$
can be obtained as limits of such measures, so that the discrete $\Phi^4_4$ model can be viewed as a limit
of block-spin models.

Recall that to show that a collection $\{X_a\}_{a \in A}$ of real-valued random variables is jointly Gaussian
it suffices to show that all joint fourth cumulants $\E_c \{X_{a_1},\ldots,X_{a_4}\}$ with $a_i \in A$
vanish. It is therefore not surprising that fourth cumulants of the spin variables play an important role
in any proof of Gaussianity for Ising-type models. In dimension $d \ge 5$, the proof in \cite{Triv2}
relied on two very important facts. First, writing $C(x,y) = \E \bigl(\sigma_x \sigma_y\bigr)$ for the 
spin correlation function, one shows that for any temperature any any Ferromagnetic interaction, one has the bound
\begin{equ}[e:treeBound]
\big|\E_c \{\sigma_{x_1},\ldots,\sigma_{x_4}\}\big| \le 2\sum_{y \in \Z^d} C(x_1,y)\cdots C(x_4,y)\;.
\end{equ}
One then observes that \textit{at the critical temperature}, the function $C$ is bounded
by 
\begin{equ}[e:boundCorr]
C(x,y) \lesssim |x-y|^{2-d}\;.
\end{equ}
Consider now four smooth compactly supported test functions $f_i$ and define
\begin{equ}
X_i = \sum_{x\in \Z^d} \sigma_x f_i(x/M)\;.
\end{equ}
In particular, the sum ranges over $\CO(M^d)$ sites. 
If one assumes that \eqref{e:boundCorr} is sharp, then one expects to have $\E X_i^2 \approx M^{d+2}$,
so that the ``correct'' normalisation for the $X_i$'s to have unit variance is expected to be 
$\xi_i = M^{-\f{d+2}2}X_i$.
On the other hand, combining the covariance bound with the bound on the fourth cumulant, a 
powercounting argument shows that 
$\E_c\{\xi_1,\ldots,\xi_4\} \lesssim M^{-2(d+2)}M^{d+8} = M^{4-d}$, which does indeed converge to $0$
as $M \to \infty$ when $d > 4$, thus showing that the $\xi_i$'s are jointly Gaussian in the limit.

Clearly this calculation does not allow us to conclude anything when $d = 4$. The main contribution of
\cite{Phi44} is to show that \eqref{e:treeBound} can actually be improved to a bound of the type
\begin{equ}[e:better]
\big|\E_c \{\sigma_{x_1},\ldots,\sigma_{x_4}\}\big| \lesssim \f{\sum_{y \in \Z^d} C(x_1,y)\cdots C(x_4,y)}
{\big(\sum_{|x| \le M}C(0,x)^2\big)^c}\;,
\end{equ}
for some (possibly very small) $c>0$. Here, one assumes that the $x_i$'s are all at distances at
least $M$ of each other. 

\begin{remark}
If one believes that the bound \eqref{e:boundCorr} represents the correct behaviour of $C$ at criticality, then
the denominator appearing in \eqref{e:better} is of order $(\log M)^c$ in dimension $4$. This however
is not known and is also not used by \cite{Phi44}, whether for deriving \eqref{e:better} or for 
deducing Theorem~\ref{theo:Gauss} from it.
\end{remark}

The proof of \eqref{e:better} relies on the ``random current'' representation of 
the Ising model in which the configuration space consists of ``currents'', namely
maps $\bn\colon E\to \N$ where $E$ denotes the set of (unoriented) nearest-neighbour pairs in
our lattice. The Ising measure then naturally leads to a weight $w$ on currents 
defined by $w(\bn) = \prod_{e \in E} \f{\beta^{\bn(e)}}{\bn(e)!}$ as well as the notion of ``source''
of a current given by
\begin{equ}
\d \bn \eqdef \Big\{x \,:\, \sum_{e \ni x} \bn(e) \,\text{is odd}\Big\}\;.
\end{equ}
The link between currents and the Ising model is the following formula. Given any finite set $A \subset \Z^d$,
one has
\begin{equ}
\E \prod_{a \in A}\sigma_a = \f{\sum_{\bn\,:\, \d\bn = A} w(\bn)}{\sum_{\bn\,:\, \d\bn = \emptyset} w(\bn)}\;.
\end{equ}
A natural notion then is that of a ``random current with source $A$'' for which the probability of seeing a
given current $\bn$ is non-vanishing only when $\d\bn = A$ in which case it is proportional to $w(\bn)$.
When $A = \{x,y\}$, a current $\bn$ with source $A$ can be interpreted (not uniquely!) as the occupation 
measure of a collection of loops
in $\Z^d$, together with a non self-intersecting path joining $x$ and $y$. In particular, the restriction of 
$\bn$ to the collection of loops connected (either directly or indirectly through other loops) to the path
joining $x$ and $y$ can be thought of as the occupation measure of one single random path joining $x$ to $y$. 

The bound \eqref{e:better} can then be reformulated in terms of intersection properties of such 
random paths. From a heuristic perspective, one gets a lot of mileage from thinking of these random paths
as simple random walk trajectories. Note that dimension $4$ is critical for the question whether the traces of two 
random walk trajectories intersect or not: in $d < 4$, the trajectories of two independent random walks
with any two starting points will intersect almost surely. In $d > 4$ on the other hand, they only intersect
with positive probability (going to $0$ as the two starting points are taken far from each other) and,
if they do, they only have a finite number of intersection points. 
In dimension $d=4$, the probability that two random walks starting at distance of order $M$ from each other do
intersect decays like $1/\log M$, but the \textit{expected} number of intersection times remains of order one
as $M \to \infty$. This shows that if they do intersect, then the number of intersection points is typically quite large,
of order $\log M$.

The bulk of the hard work performed in \cite{Phi44} is to show that the random paths arising in the
 random cluster representation of the Ising model at criticality exhibit a similar behaviour, but with $\log M$
 replaced by some quantity of size at least $(\log M)^c$ for some $c > 0$.
The argument is a masterpiece combining a delicate multiscale analysis, topological arguments, 
and probabilistic reasoning.
One of the main problem the authors have to overcome is the fact that these random paths
are \textit{very} far from being simple random walks and only satisfy some spatial version of the Markov 
property.

\section{Rotational invariance for the critical FK models}

As already mentioned a number of times, a crucial feature of $2d$ equilibrium statistical mechanics is
the fact that most models are expected to obey a form of conformal invariance (or equivariance)
when considering large-scale observables at the critical temperature. 
This expectation and the resulting link to the well understood world of $2d$ conformal field theories 
allows to generate a plethora of conjectures regarding the large-scale behaviour of these models, but these
are in many cases extremely hard to prove. Consider for example the $N$-step
$2d$ self-avoiding 
random walk which is simply the uniform measure on all functions $h \colon \{0,\ldots,N\} \to \Z^2$
such that $h(0) = 0$ and such that $|h(i+1) - h(i)| = 1$ for all $i < N$. Exploiting the expected
conformal invariance of its suitably rescaled large-$N$ limit, one expects the size of $h(N)$
to be of order $N^{3/4}$ and its rescaling by $N^{3/4}$ to converge to a specific continuous
random curve, namely SLE$_{8/3}$ \cite{LSW}. Rigorously, almost \textit{nothing} non-trivial is known:
although the diameter of the range of $h$ trivially has to be at least $\sqrt{N/\pi}$, the current
best lower bound on the endpoint does not even match that! Instead, one only knows 
the bound $(\E |h(N)|^p)^{1/p} \ge \f16 N^{p/(2p+2)}$ that was recently obtained by 
Madras \cite{Madras}. Similarly, while one trivially 
has $|h(N)| \le N$, the best non-trivial upper bound is pretty much the weakest
possible improvement, namely that for every $p \ge 1$ one has 
$\lim_{N \to \infty} N^{-1}(\E |h(N)|^p)^{1/p} = 0$, obtained around the same time
by Duminil-Copin and Hammond \cite{SAWUp}. One main obstruction is that there is at the moment
no proof showing that the self-avoiding random walk is conformally invariant at large scales.

While this illustrates the importance of showing that statistical models are conformally
invariant (or at least rotationally invariant as a crucial first step) at criticality, the 
strategy of proof for such claims has so far mostly relied on finding a large enough
collection of observables that already satisfy a discrete analogue of conformal invariance,
typically by solving a discrete analogue of the Cauchy--Riemann equations. See for example
Chelkak and Smirnov's proof of conformal invariance for the Ising model on isoradial
graphs \cite{MR2957303} and Smirnov's proof of conformal invariance
for critical percolation \cite{Stas}.
The two-dimensional FK model with $q \le 4$ already mentioned in Section~\ref{sec:Disc} is one of the simplest 
models where conformal invariance at criticality is expected, but where it is not known how to obtain this
from a suitable discrete conformal invariance.
In the recent work \cite{Rot}, Duminil-Copin et al.\ show that the large-scale behaviour of these
models is indeed rotationally invariant.

To define the notion of ``large-scale behaviour'', we recall that the configuration space of the 
FK model is the same as that for regular percolation, see Figure~\ref{fig:perco}. Such a configuration
can alternatively be described as a collection of non self-intersecting loops separating the percolation
clusters from the clusters of the dual configuration. (Actually it naturally yields 
\textit{two} collections of loops, depending on whether the loop encloses a percolation cluster 
of the primary or of the dual configuration, but we will ignore this detail for the 
sake of our exposition.) Given two collections $\CF$ and $\bar \CF$ of
non self-intersecting loops in the plane, one then defines a distance between them in the following way.
Given (small) $\eta > 0$, write $\CB_\eta \subset \R^2$ for a large chunk of a fine lattice in $\R^2$,
for example $\CB_\eta = \eta \Z^2 \cap [-\eta^{-1},\eta^{-1}]^2$.
Given a loop $\gamma$ and assuming that its image doesn't intersect the set $\CB_\eta$, 
one then denotes by $[\eta]_\gamma$ its homotopy class in $\R^2 \setminus \CB_\eta$.
One then postulates that $d_H(\CF,\bar \CF) \le \eta$  if and only if, for every
$\gamma \in \CF$ that encloses at least two elements of $\CB_\eta$ but not all of it, 
there exists $\bar \gamma \in \bar \CF$ such that $[\gamma]_\eta = [\bar\gamma]_\eta$
and vice-versa. (The $H$ here stands for `homotopy'.)

Given a metric space $(M,d)$, the metric $d$ lifts naturally to a metric on the space 
of probability measures on $M$ which metrises the topology of weak convergence (at least when
$M$ is ``nice'', for example Polish). This is done by considering the Wasserstein 
(also sometimes called Kantorovich--Rubinstein or Monge--Kantorovich) distance
\begin{equ}
d(\mu,\nu) = \inf_{\P \in \CC(\mu,\nu)} \int d(x,y)\,\P(dx,dy)\;,
\end{equ}
where $\CC(\mu_1,\mu_2)$ denotes the set of all couplings between $\mu_1$ and $\mu_2$,
that is probability measures on $M^2$ 
with $i$th marginal equal to $\mu_i$. Note that with this definition, the map that assigns
to $x$ the probability measure $\delta_x$ concentrated at $x$ is an isometry. 

Fix now once and for all $q \in [1,4]$ and consider
a smooth bounded simply connected domain $\Omega \subset \R^2$. For $\eps > 0$,
write $\P_{\eps,\Omega}$ for the critical FK measure (viewed as a measure on collections of loops) 
on $\eps \Z^2 \cap \Omega$ with free boundary conditions. We also write $\P_{\eps}$ for the 
limit of $\P_{\eps,\Omega}$ as $\Omega \to \R^2$. Given an angle $\theta \in \R$,
we also write $R_\theta$ for the rotation by $\theta$, which naturally acts on loops in $\R^2$.
The large-scale rotational
invariance of the critical FK model can then be formulated as follows.

\begin{theorem}\label{theo:FKrot}
For every domain $\Omega \subset \R^2$ as above and every angle $\theta$ one has 
\begin{equ}
\lim_{\eps \to 0} d_H\big(R_\theta^*\P_{\eps,\Omega}, \P_{\eps,R_\theta\Omega}\big) = 0\;.
\end{equ}
Furthermore, one has $\lim_{\eps \to 0} d_H(R_\theta^*\P_{\eps}, \P_{\eps}) = 0$.
\end{theorem}

\begin{figure}
\begin{center}
  \begin{tikzpicture}[scale=0.55]
    \clip(0.8,-0.1) rectangle (13.2,8);
    % loop over the lattice points
  	\pgfmathsetmacro{\oldHor}{0}
  	\pgfmathsetmacro{\oldVert}{0}
  	\pgfmathtruncatemacro{\j}{0}
    
    \draw[very thin,draw=black!10] (0,0) -- (14,0);
    
    \pgfplotsforeachungrouped \angle in {0,10,-5,20,15,-12,-10,0} {
  		\pgfmathsetmacro{\vertPos}{\oldVert+cos(\angle)}
	  	\pgfmathsetmacro{\horPos}{\oldHor+sin(\angle)}
	  	\pgfmathtruncatemacro{\j}{\j+1}
      \foreach \i in {0,...,14}{
	  	\pgfmathtruncatemacro{\sumCoor}{\j+\i}
      	\ifodd\sumCoor
			\gdef\thisColor{black}
			\gdef\otherColor{white}
	        \draw (\horPos+\i,\vertPos) -- (\oldHor+\i+1,\oldVert);
	        \draw (\horPos+\i,\vertPos) -- (\oldHor+\i-1,\oldVert);
		\else
			\gdef\thisColor{white}
			\gdef\otherColor{black}
		\fi
        \draw[very thin,draw=black!10] (\horPos+\i,\vertPos) -- (\horPos+\i+1,\vertPos);
        \draw[very thin,draw=black!10] (\horPos+\i,\vertPos) -- (\oldHor+\i,\oldVert);

        \draw[fill=\thisColor] (\horPos+\i,\vertPos) circle(2pt);
        \draw[fill=\otherColor] (\oldHor+\i,\oldVert) circle(2pt);
      };
	  	\pgfmathsetmacro{\oldHor}{\horPos}
	  	\pgfmathsetmacro{\oldVert}{\vertPos}
	};
  \end{tikzpicture}
  \hspace{5mm}
  \begin{tikzpicture}[scale=0.55]
    \clip(2.3,-0.1) rectangle (10.2,8);
    % loop over the lattice points
  	\pgfmathsetmacro{\oldHor}{0}
  	\pgfmathsetmacro{\oldVert}{0}
    
    \draw[very thin,draw=black!10] (0,0) -- (14,0);
    \def\angle{20}
    \pgfplotsforeachungrouped \j in {1,...,10} {
  		\pgfmathsetmacro{\vertPos}{\oldVert+cos(\angle)}
	  	\pgfmathsetmacro{\horPos}{\oldHor+sin(\angle)}
      \pgfplotsforeachungrouped \i in {0,...,10}{
	  	\pgfmathtruncatemacro{\sumCoor}{\j+\i}
      	\ifodd\sumCoor
			\gdef\thisColor{black}
			\gdef\otherColor{white}
	        \draw (\horPos+\i,\vertPos) -- (\oldHor+\i+1,\oldVert);
	        \draw (\horPos+\i,\vertPos) -- (\oldHor+\i-1,\oldVert);
		\else
			\gdef\thisColor{white}
			\gdef\otherColor{black}
		\fi
        \draw[very thin,draw=black!10] (\horPos+\i,\vertPos) -- (\horPos+\i+1,\vertPos);
        \draw[very thin,draw=black!10] (\horPos+\i,\vertPos) -- (\oldHor+\i,\oldVert);

        \draw[fill=\thisColor] (\horPos+\i,\vertPos) circle(2pt);
        \draw[fill=\otherColor] (\oldHor+\i,\oldVert) circle(2pt);
        \ifnum \i=3 \ifnum \j=3
	  		\pgfmathsetmacro{\xx}{\horPos+\i}
	  		\pgfmathsetmacro{\yy}{\vertPos}
        \fi\fi
      };
	  	\pgfmathsetmacro{\oldHor}{\horPos}
	  	\pgfmathsetmacro{\oldVert}{\vertPos}
	};
	\draw[red!80!black] (\xx,\yy) -- ++(45-\angle/2:8);
	\draw[red!80!black] (\xx,\yy) -- ++(225-\angle/2:7);
  \end{tikzpicture}
\end{center}
\caption{Examples of graphs $L(\alpha)$. On the left is a generic $\alpha$ while on the right $\alpha$ is constant but non-zero. The graph itself is drawn in black, the vertices of its dual graph are drawn in white, and the associated diamond graph is light gray. In red, we draw one of the symmetry axes of the second graph.}\label{fig:lattice}
\end{figure}
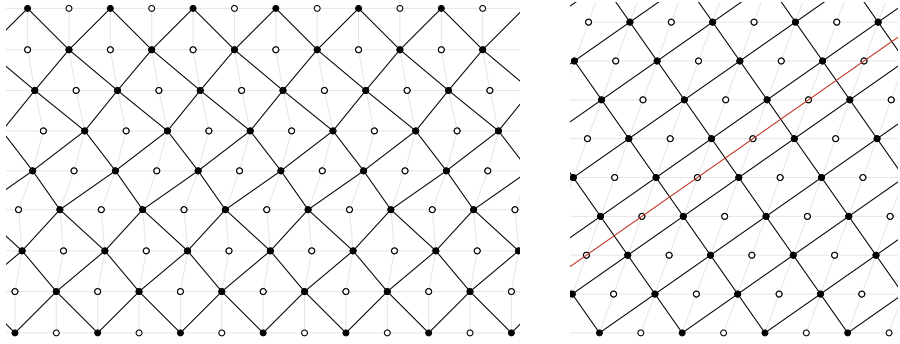

We only focus on the second statement since it turns out that the first one can be deduced from it
without too much effort.
In fact, the authors of \cite{Rot} show a type of universality statement 
for the FK model on rectangular lattices, but its formulation requires some preparation. 
We start by defining a specific class of isoradial embeddings of the two-dimensional 
square lattice into the plane.
Recall that a planar graph embedded in the plane is isoradial if, for each face $f$, 
there exists a circle of radius $1$ containing all the vertices of $f$. (For example, the 
canonical embedding of the square lattice is isoradial.)

Given a bi-infinite sequence $\alpha \colon \Z \to (-\f\pi2,\f\pi2)$, we consider the map $\iota_\alpha \colon \Z^2 \to \R^2$ given by
\begin{equ}
\iota_\alpha\colon (x,y) \mapsto \bigl(x + s_y,c_y\bigr)\;,\qquad s_y = \sum_{k\in (0,y]} \sin(\alpha_k)\;,\quad
c_y = \sum_{k\in (0,y]} \cos(\alpha_k)\;,
\end{equ}
with the convention that for $y<0$, $\sum_{(0,y]} = -\sum_{(y,0]}$.
This defines an isoradial graph $L(\alpha)$ by considering the embedding of 
$\{(x,y)\,:\, x+y \,\text{even}\}$ (joined by diagonal edges) under $\iota_\alpha$ (see Figure~\ref{fig:lattice}).
The dual graph $L^*(\alpha)$ of $L(\alpha)$ is then given by the embedding of 
$\{(x,y)\,:\, x+y \,\text{odd}\}$. The associated ``diamond graph'' has as its vertices 
both the vertices of $L(\alpha)$ and the centres of its faces, and its edges are given by all pairs 
$(v,f)$ with $v$ a vertex and $f$ a face such that $v \in f$. 
The diamond graph is simply given by the embedding of the usual lattice $\Z^2$
with nearest-neighbour edges under $\iota_\alpha$.

It is crucial at this stage to note that the critical FK model on $L(\alpha)$ is \textit{not} given
by simply pushing forward the critical FK model on $\Z^2$ under the map $\iota_\alpha$. Instead,
one reweighs each edge of the graph in a very specific way that depends on the length of the edge.
More specifically, viewing a configuration of the FK model as a subset $\omega \subset E$ 
of the set of edges of the (finite) graph on which the model is considered, the probability of 
seeing a given configuration $\omega$ is proportional to 
\begin{equ}[e:defFK]
\Big(\prod_{e \in \omega}p_e\Big)\Big(\prod_{e \in E\setminus \omega}(1-p_e)\Big) q^{k(\omega)}\;,
\end{equ}
where $k(\omega)$ denotes the number of connected components of the subgraph $\omega$.
The formula for $p_e$ as a function of $q$ and the length of the edge $e$ is explicit but not
relevant for the sake of this discussion.

The most important 
step in the proof is to show that the large-scale connectivity properties of 
the critical FK model on $L(\alpha)$ are very close to those of the model on $L(T_j \alpha)$, where
$T_j$ swaps the $j$th and $(j+1)$th component:
\begin{equ}
(T_j \alpha)_k = 
\left\{\begin{array}{cl}
	\alpha_{j+1} & \text{if $k=j$,} \\
	\alpha_{j} & \text{if $k=j+1$,} \\
	\alpha_k & \text{otherwise.}
\end{array}\right.
\end{equ}
Furthermore, there exists a natural coupling between the FK measures on the two lattices
which implements this ``closedness''.
This part of the proof exploits the link to the six vertex model and its ``solvability''
using the transfer matrix formalism. One then deduces from this that the model on the
standard lattice $L(0)$ is very close to that on a rotated rectangular lattice $L(\alpha)$
with $k \mapsto \alpha_k$ constant (see the right half of Figure~\ref{fig:lattice}).
This works by fixing some large $N>0$ (which is then eventually sent to infinity) 
and starting from $\alpha^{(i)}_k = \alpha \one_{k \ge N}$
and then swapping components in such a way as to move some of the non-zero components down 
until une ends up with $\alpha^{(f)}_k = \alpha (\one_{|k| \le N} + \one_{k > 3N})$.
Since one has $L(0) \approx L(\alpha^{(i)})$ and $L(\alpha) \approx L(\alpha^{(f)})$, the desired
statement follows if one can control the error made at each step of the argument.
This turns out to be extremely delicate and one has to exploit subtle stochastic 
cancellations along the way. One trick is to allow the vertices of the set $\CB_\eta$ around
which the homotopy classes are computed to move a little bit with each application of 
a swapping operator $T_j$ and to show that this motion ends up being diffusive (and therefore ``slow'')
rather than ballistic.

Once one knows that $\lim_{\eps \to 0}d_H(\P_{\eps,L(0)}, \P_{\eps,L(\alpha)}) = 0$, 
the second part of Theorem~\ref{theo:FKrot} follows at once. The idea is simply to note that
$L(\alpha)$ is invariant under reflection along a line with angle $\f{\pi}4-\f\alpha2$,
but that the effect of this reflection on $L(0)$ is the same as that of a rotation by angle $\alpha$
(since it is itself invariant under reflection along a line with angle $\f\pi4$),
so that 
\begin{equ}
d_H(\P_{\eps}, R_\alpha^*\P_{\eps})
\le d_H(\P_{\eps,L(0)}, \P_{\eps,L(\alpha)}) + d_H(\P_{\eps,L(\alpha)},R_\alpha^*\P_{\eps,L(0)})
= 2 d_H(\P_{\eps,L(0)}, \P_{\eps,L(\alpha)})\;,
\end{equ}
and the claim follows.

\begin{funding}
This work was partially supported by the Royal Society through a research professorship.
\end{funding}

%------
% Insert the bibliography.
%------
\newcommand{\newblock}{}
\bibliographystyle{emss}
\bibliography{refs}

\end{document}